\documentclass[reqno,12pt]{amsart}
\usepackage{amsmath,amsfonts,amssymb,amscd,amsthm,amsbsy,bm,latexsym}
\usepackage{gensymb}
\usepackage[pdftex,bookmarksnumbered=true,bookmarksopen=true,          %bookmark
            colorlinks=true,pdfborder=001,citecolor=blue,
            linkcolor=red,anchorcolor=green,urlcolor=blue]{hyperref}
\usepackage{graphicx}
\allowdisplaybreaks

\usepackage{color}

\newcommand{\gm}{\gamma}
\newcommand{\R}{\mathbb{R}}

\newcommand{\eps}{\varepsilon}
\newcommand{\Proof}{\begin{proof}}
\newcommand{\End}{\end{proof}}

\newtheorem{lemma}{Lemma}[section]
\newtheorem{theorem}{Theorem}[section]
\newtheorem{definition}{Definition}[section]
\newtheorem{proposition}{Proposition}[section]
\newtheorem{remark}{Remark}[section]

\newtheorem{example}{Example}[section]
\numberwithin{equation}{section}

\begin{document}
%  \bibliography{bibfile}

\title[EXPONENTIAL CONVERGENCE OF THE 1-GRAPH]{EXPONENTIAL CONVERGENCE OF 1-GRAPH OF THE SOLUTION SEMIGROUP OF CONTACT HAMILTON-JACOBI EQUATIONS}

\author[L. Jin]{Liang Jin}
\address[L. Jin]{Academy of Mathematics and Systems Science, CAS,
Beijing 100190, China}
\email{jinliang@amss.ac.cn}

\author[L. Wang]{Lin Wang}
\address[L. Wang]{Yau Mathematical Sciences Center, Tsinghua University, Beijing 100084, China}
\email{lwang@math.tsinghua.edu.cn}

\date{\today}

\begin{abstract}
  Under certain assumptions, we show that for the solution semigroup of evolutionary contact Hamilton-Jacobi equations, its 1-graph, as a pseudo Legendrian graph, converges {\it exponentially} to the 1-graph of the viscosity solution of stationary equations in the sense of certain Hausdorff metrics. This result reveals an essential difference between certain dissipative systems and conservative systems from weak KAM aspects.
\end{abstract}

\keywords{contact Hamilton-Jacobi equation, solution semigroup, 1-graph, exponential convergence}
\subjclass[2010]{35D40, 35F21,  37J50}

\maketitle
\thispagestyle{empty}
\tableofcontents

%================================================================================================Sect. 1
%================================================================================================Sect. 1
%================================================================================================Sect. 1
\section{Introduction}
Let $M$ be a  connected, closed $C^{\infty}$ Riemannian manifold, we consider a $C^r (r\geq3)$ function $H:T^{*}M\times\R\rightarrow\R$ called a contact Hamiltonian that satisfies
\begin{itemize}
  \item [\textbf{(H1)}] \textbf{Positive Definiteness}. For every $(x,u,p)\in T^{\ast}M\times\R$, the second partial derivative $\partial^2 H/\partial p^2(x,u,p)$ is positive definite as a quadratic form,
  \item [\textbf{(H2)}] \textbf{Superlinear Growth}. For every $(x,u)\in M\times\R$, $H(x,u,p)$ is superlinear with respect to $p$, that is
                        \begin{equation*}
                        \lim_{|p|_{x}\rightarrow\infty}\frac{H(x,u,p)}{|p|_{x}}\rightarrow\infty,
                        \end{equation*}
                        where $|\cdot|_{x}$ denotes the norm on $T_{x}^{\ast}M$ induced by the Riemannian metric,
  \item [\textbf{(H3)}] \textbf{Moderate increasing}. There is constant $\Lambda>0$ such that for every $(x,u,p)\in T^{\ast}M\times\R$,
                        \begin{equation*}
                        0<\frac{\partial H}{\partial u}(x,u,p)\leq\Lambda.
                        \end{equation*}
\end{itemize}
We focus on the following two first order partial differential equations  associated to $H$, namely
\begin{equation}\label{evo}
\begin{cases}
\partial_tu(t,x)+H(x,u(t,x),\partial_xu(t,x))=0,\hspace{0.3cm}(t,x)\in[0,+\infty)\times M,\\
u(0,x)=\varphi(x),
\end{cases}
\end{equation}
and
\begin{equation}\label{sta}
H(x,u(x),d_{x}u(x))=0,\hspace{0.3cm}x\in M.
\end{equation}

If $\frac{\partial H}{\partial u}\equiv0$, (\ref{evo}) and (\ref{sta}) are reduced to classical Hamilton-Jacobi equations. In this case, there are a broad class of works on the convergence of viscosity solutions of evolutionary equation (\ref{evo}) to viscosity solutions of stationary equation (\ref{sta}) from both dynamical and PDE approaches, see \cite{ds,F22,F3,FaM,II,NR,WY} and references therein. Besides the convergence of the viscosity solution itself, M. Arnaud found that such convergence can be viewed geometrically as the convergence of the differential of the Lax-Oleinik semigroup as a pseudo Lagrangian graph in the sense of Hausdorff metric \cite{Ana}.

If $\frac{\partial H}{\partial u}\not\equiv 0$, (\ref{evo}) and (\ref{sta}) are called contact Hamilton-Jacobi equations. From the view of physics, these equations (\ref{evo}) and (\ref{sta}) appear naturally in contact Hamiltonian mechanics \cite{BCT1}, which is the most natural extension of Hamiltonian mechanics \cite{Ar1,Ar}. In a recent work \cite{SWY}, X. Su, L. Wang and J. Yan showed that under the assumptions (H1)-(H3), there exists an interval $\mathcal{C}_{H}\subseteq\mathbb{R}$ only depending on $H$ (see Appendix B) such that if $0\in\mathcal{C}_{H}$, then for every $\varphi(x)\in C^0(M,\R)$, the unique viscosity solution $u(t,x):=T_{t}\varphi(x)$ of \eqref{evo}  converges to a viscosity solution $u_{-}(x)$ of \eqref{sta} in the $C^0$-norm as $t$ goes to infinity, i.e.,
\begin{equation}\label{c}
\|T_{t}\varphi-u_{-}\|_{C^{0}}\rightarrow0\,\,\,\,\text{as}\,\,\,t\rightarrow\infty,
\end{equation}
where $T_{t}$ is referred as a generalized Lax-Oleinik semigroup, named {\it solution semigroup}, see (\ref{uxt}) below. More recently, a similar convergence result was obtained by X. Li \cite{Li} using generalized dynamical and PDE techniques.

It is well known that assumptions (H1)-(H3) implies for $t>0$, $T_{t}\varphi$ and $u_{-}$ are Lipschitz functions on $M$. Given a Lipschitz function $u:M\rightarrow\mathbb{R}$, let $\mathcal{D}_{u}$ denotes the set of differentiable point of $u$, a compact subset of $T^{\ast}M\times\mathbb{R}$ called 1-graph of $u$ is defined as
\begin{equation}\label{def-j}
\bar{\mathcal{J}}^{1}_{u}:=\overline{\{(x,u(x),d_{x}u(x)):x\in\mathcal{D}_{u}\}}.
\end{equation}
Geometrically, $\bar{\mathcal{J}}^{1}_{u}$ can be considered as a pseudo Legendrian graph over $M$.  Let us recall
\begin{definition}\label{hd}
Let $(X,d)$ be a metric space and $\mathcal{K}(X)$ be the set of non-empty compact subset of $X$. The Hausdorff metric $d_{H}$ induced by $d$ is defined by
\begin{equation}
d_{H}(K_{1},K_{2})=\max\bigg\{\max_{x\in K_{1}}d(x,K_{2}),\max_{x\in K_{2}}d(x,K_{1})\bigg\},\,\,\,\forall K_{1},K_{2}\in\mathcal{K}(X).
\end{equation}
\end{definition}

In this note, we will study the rate of convergence formed as  \eqref{c} from a more quantitative and geometrical point of view.   The characteristics of (\ref{evo}) is formulated as contact Hamilton equations (see (\ref{hjech}) below). It follows from (H3) that the energy along the local contact flow is dissipative. Comparably, regarding the rate of convergence of Lax-Oleinik operators for classical Hamilton-Jacobi equations, due to the energy conservation, the exponential convergence of evolutionary solution for classical Hamiltonian requires delicate dynamical conditions, it could not be achieved even if the Aubry set of the corresponding Lagrangian system consists of only one hyperbolic periodic orbit, see \cite{WY10,WY11}.
By the weak KAM approaches developed for contact Hamilton-Jacobi equations, we will show that the exponential rate of convergence can be achieved for both the solution semigroup and its 1-graph under moderate increasing assumptions proposed on the contact Hamiltonians. More precisely, we obtain
\begin{theorem}\label{main}
Let $H:T^{\ast}M\times\R\rightarrow\R$ be a contact Hamiltonian satisfying (H1)-(H3), if $0\in\mathcal{C}_{H}$, then \eqref{sta} admits a unique viscosity solution $u_{-}(x)$ and for every $\varphi(x)\in C^0(M,\R)$, there exists $\lambda:=\lambda(H)>0$ such that as $t\rightarrow\infty$,
\begin{equation}\label{ecs}
\|T_{t}\varphi-u_{-}\|_{C^{0}}\leq O(e^{-\lambda t}),
\end{equation}
\begin{equation}\label{ecj}
d_{H}(\bar{\mathcal{J}}^{1}_{T_{t}\varphi},\bar{\mathcal{J}}^{1}_{u_{-}})\leq O(e^{-\frac{\lambda}{3}t}),
\end{equation}
where $d_{H}$ denotes the Hausdorff metric induced by any Riemannian metric on $T^{\ast}M\times\R$.
\end{theorem}

This result reveals an essential difference between certain dissipative systems and conservative systems from weak KAM aspects.

From dynamical point of view, we have to consider the second-order derivatives of certain global minimizing curves $\gm$ in order to achieve (\ref{ecj}). Roughly speaking, $\frac{\lambda}{3}$ can be implied by boundedness of $\ddot{\gm}$. Nevertheless, this exponent could be improved if the exponential decay of $\ddot{\gm}$ happens, for instance,  we have $\ddot{\gm}\equiv 0$ for the integrable contact Hamiltonian $H(x,u,p):=\lambda u+\frac{1}{2}|p|^2$ with nonzero-constant initial values. In this case, both (\ref{ecs}) and (\ref{ecj}) have exactly the same convergence rate $O(e^{-\lambda t})$.

%From dynamical point of view, the characteristics of (\ref{evo}) is formulated as contact Hamilton equations (see (\ref{hjech}) below). From (H3), the local contact Hamilton flow $\Phi_H^t$ compress the volume form with the rate $e^{-\mu_t}$, where $\mu_t:=\int_0^t\frac{\partial H}{\partial u}(\Phi_H^s)ds$. In particular, if $\frac{\partial H}{\partial u}$ has a positive lower bound, then the volume form is compressed uniformly, from which it is natural to expect the exponential convergence as Theorem \ref{main}. Consequently, the dynamics generated by these Hamiltonians seem to be lack of much abundance as classical Hamiltonian systems exhibit \cite{MS}. Under the assumption (H3), however, we only have the non-uniform compression due to the non-compactness of $u$-component. Fortunately, Theorem \ref{main} can be still achieved by {\it a priori} compactness  of certain global minimizing orbits based on weak KAM theory.

It is worth mentioning that the convergence result \eqref{c} can be obtained under (H1)-(H2) and the assumption that $0\leq\frac{\partial H}{\partial u}\leq\Lambda$, under which $\frac{\partial H}{\partial u}=0$ could happen on subsets of $T^*M\times\R$. However, for Theorem \ref{main}, the assumption $\frac{\partial H}{\partial u}>0$ is necessary as the following example shows.

\begin{example}\label{ex}
Let $M=\mathbb{T}$ and the contact Hamiltonian $H_0:T^\ast\mathbb{T}\times\R\rightarrow\R$ is defined as
\begin{equation}
H_0(x,u,p):=\frac{1}{2}(p^2+\rho(u^3)),
\end{equation}
where $\rho:\R\rightarrow\R$ is a $C^\infty$ increasing function that equals to identity on $[-1,0]$ and takes constant values on $(-\infty,-2]$ and $[1,\infty)$.
\end{example}

It is easy to see that $H_0$ satisfies the assumptions (H1)-(H2) and $0\leq\frac{\partial H}{\partial u}\leq\Lambda$. In particular, $\frac{\partial H}{\partial u}(x,0,p)=0$. For the initial value $\varphi(x)\equiv-1$, a direct calculation (see Appendix A) shows that the function
\begin{equation*}
T_t\varphi(x)=-(1+t)^{-\frac{1}{2}}, \quad \forall x\in M,
\end{equation*}
is the unique viscosity solution of \eqref{evo}. Note that $u_{-}(x)\equiv0$ is a viscosity solution of \eqref{sta}, we have \[\|T_t\varphi-u_{-}\|_{C^{0}}\sim O\left(\frac{1}{t^{2}}\right), \quad {\rm as}\quad  t\rightarrow\infty.\]

This note is outlined as follows. In Section 2, we fix some notations and then recall the necessary results as preliminaries. A key lemma connecting the convergence rate of the solution semigroup and the one of its 1-graph is provided in Section 3, from which the proof of Theorem \ref{main} is completed in Section 4. Four appendices are included as supplemental materials.

\section{Notations and preliminaries}
In this section, we shall fix the notations used in this note once and for all, and then present some results that are necessary for later sections. Some results are collected from previous works for which the proof is omitted, others results are formulated and proved in detail. We shall try our best to make the presentation available to almost everyone.

\subsection{Generalities}
Let $M$ be a connected, closed $C^{\infty}$ Riemannian manifold. We will denote $(x,\dot{x})$ a point of the tangent bundle $TM$ with $\dot{x}\in T_{x}M$, $(x,p)$ a point of the cotangent bundle $T^{\ast}M$ with $p\in T^{\ast}_{x}M$ and for any $x\in M$, $\langle\cdot,\cdot\rangle$ the canonical duality between $T^{\ast}_{x}M$ and $T_{x}M$. We will denote $z$ or, in component fashion, $(x,u,p)$ a point in $J^{1}(M,\R)\cong T^{\ast}M\times\R$, the space of 1-jets over $M$, $(x,u,\dot{x})$ a point in $TM\times\R$. Let $\pi_{1}:T^{\ast}M\times\R\rightarrow M$ be the projection $(x,u,p)\mapsto x$ and $\pi:T^{\ast}M\times\R\rightarrow T^{\ast}M$ the projection $(x,u,p)\mapsto(x,p)$ which forgets the $u$-component.

We now focus on the Riemannian metric $g$ defined on $M$ and its derivatives. We abuse to denote $|\cdot|_{x}$ the norm on $T^{\ast}_{x}M$ or $T_{x}M$ induced by $g$ since there is no need to clarify them and $d_{M}$ be the metric on $M$ induced by $g$. By elementary Riemannian geometry, there is a canonical Riemannian structure $\tilde{g}$ on $TM$ induced by $g$, see \cite{dC} for details. We denote $d_{TM}$ the metric on $TM$ induced by $\tilde{g}$ and notice that
\begin{itemize}
  \item for two tangent vectors $\dot{x},\dot{x}^{\prime}\in T_{x}M$,
        \begin{equation}\label{dtm-1}
        d_{TM}((x,\dot{x}),(x,\dot{x}^{\prime}))=|\dot{x}-\dot{x}^{\prime}|_{x}.
        \end{equation}

  \item for $(x_{1},\dot{x}_{1}),(x_{2},\dot{x}_{2})\in TM$,
        \begin{equation}\label{dtm-2}
        d_{M}(x_{1},x_{2})\leq d_{TM}((x_{1},\dot{x}_{1}),(x_{2},\dot{x}_{2})).
        \end{equation}
\end{itemize}
Let $|\cdot|$ denote the usual norm on $\R$, we define canonical metrics on $TM\times\R$ as for $(x_{1},u_{1},\dot{x}_{1})$ and $(x_{2},u_{2},\dot{x}_{2})\in TM\times\R$, $d_{TM\times\R}=d_{TM}((x_{1},\dot{x}_{1}),(x_{2},\dot{x}_{2}))+|u_{1}-u_{2}|$.

\subsection{Functional setting}
Let $C^{0}(M,\R)$ be the Banach space of all continuous functions on $M$ with the usual $C^{0}$ norm $\|u\|_{C^{0}}:=\max_{x\in M}|u(x)|$ and $C^{\infty}(M,\R)$ the space of all smooth functions on $M$. For $\kappa>0$, we say a function $u:M\rightarrow\mathbb{R}$ is $\kappa$-Lipschitz continuous if for any $x_{1},x_{2}\in M$, $|u(x_{1})-u(x_{2})|\leq\kappa d_{M}(x_{1},x_{2})$. We say $u$ is a Lipschitz function if there is $\kappa>0$ such that $u$ is $\kappa$-Lipschitz continuous. The same notations apply to the case when $M$ is replaced by other $C^{\infty}$ manifolds.

For a Lipschitz function $u:M\rightarrow\R$, we denote $\mathcal{D}_{u}\subseteq M$ the set of differentiable points of $u$. A co-vector $p\in T^{\ast}_{x}M$ is called a reachable differential of $u$ at $x$ if there exists a sequence $\{x_{k}\}_{k\in\mathbb{N}}\subset\mathcal{D}_{u}\setminus\{x\}$ with $\lim_{k\rightarrow\infty}x_{k}=x$ such that $\lim_{k\rightarrow\infty}d_{x}u(x_{k})=p$. By Rademacher's theorem, $\overline{\mathcal{D}_{u}}=M$ and we define
\begin{itemize}
  \item $D^{*}u(x)\subseteq T^{\ast}_{x}M$ the set of all reachable differentials of $u$ at $x$,

  \item a compact subset of $T^{\ast}M$ called the graph of differential of $u$ by
        \begin{equation}\label{def-g}
        \bar{\mathcal{G}}_{u}:=\overline{\{(x,d_{x}u(x)):x\in\mathcal{D}_{u}\}}.
        \end{equation}
\end{itemize}
It follows that for any $x\in M$, $D^{*}u(x)=\bar{\mathcal{G}}_{u}\cap T^{\ast}_{x}M$ is nonempty and compact and by \eqref{def-j},  $\pi_{1}\bar{\mathcal{J}}^{1}_{u}=M, \pi\bar{\mathcal{J}}^{1}_{u}=\bar{\mathcal{G}}_{u}$.

We introduce the notion of locally semiconcavity on a Riemannian manifold and directional derivative, then present an useful lemma related to them.
\begin{definition}\label{semiconcave}
Let $\mathcal{O}$ be an open subset of a Riemannian manifold $M$, a function $u:\mathcal{O}\rightarrow\R$ is said to be semiconcave if there exists a nondecreasing, upper semicontinuous function $\omega:\R^{+}\rightarrow\R^{+}$ such that
\begin{enumerate}
  \item $\omega(r)=o(r)$ as $r\rightarrow0^{+}$,
  \item for any constant-speed geodesic path $\gamma(t), t\in[0,1]$, whose image is included in $\mathcal{O}$,
        \begin{equation}\label{semiconcave ineq}
        (1-t)u(\gamma(0))+tu(\gamma(1))-u(\gamma(t))\leq t(1-t)\omega(d_{M}(\gamma(0),\gamma(1))),
        \end{equation}
        where $d$ is the induced distance on $M$.
\end{enumerate}

A function $u:M\rightarrow\mathbb{R}$ is said to be locally semiconcave if for each $x\in M$ there is a neighborhood $\mathcal{O}$ of $x$ in $M$ such that \eqref{semiconcave ineq} holds true provided $\gamma(0),\gamma(1)\in\mathcal{O}$.
\end{definition}

The following property shows that locally semiconcave assumption implies much more than only being locally Lipschitzian, see \cite[Theorem 3.2.1; Theorem 3.3.6]{CS}.
\begin{lemma}\label{dd}
Let $u:M\rightarrow\R$ be a locally semiconcave function, then for any $x\in M$ and a $C^{1}$ curve $\gamma:[0,\sigma]\rightarrow M$ with $\gamma(0)=x,\dot{\gamma}(0)=\dot{x}$, the directional derivative
\begin{equation*}
\partial u(x,\dot{x}):=\lim_{h\rightarrow0^{+}}\frac{u(\gamma(h))-u(x)}{h}
\end{equation*}
exists and
\begin{equation}\label{dd representation}
\partial u(x,\dot{x})=\min_{p\in D^{*}u(x)}\langle p,\dot{x}\rangle.
\end{equation}
Moreover, let $\eta:[a,b]\rightarrow M$ be any $C^{1}$ curve, then
\begin{equation}\label{NL}
u(\gamma(b))-u(\gamma(a))=\int_{a}^{b}\partial u(\gamma(s),\dot{\gamma}(s))ds.
\end{equation}
\end{lemma}

\begin{remark}\label{dd-r}
From \cite{CS}, it is easy to see that
\begin{enumerate}
  \item the directional derivative $\partial u(x,\dot{x})$ does not depend on the $C^{1}$ curves representing the velocity vector $(x,\dot{x})$.
  \item  the viscosity solutions of (\ref{evo}) and (\ref{sta}) are locally semiconcave.
\end{enumerate}
\end{remark}

Let $H:T^{\ast}M\times\R\rightarrow\R$ be a $C^{r}$ contact Hamiltonian satisfies (H1)-(H3). As usual, the Lagrangian $L(x,u,\dot{x})$ associated to $H(x,u,p)$ is defined by
\[L(x,u, \dot{x}):=\sup_{p\in T_x^*M}\{\langle p,\dot{x}\rangle-H(x,u,p)\}.\]
From (H1)-(H3), it follows that the Lagrangian $L(x,u,\dot{x})$ satisfies
\begin{itemize}
  \item [\textbf{(L1)}]  \textbf{Positive Definiteness}. For every $(x,u)\in M\times\R$, the second partial derivative $\partial^2 L/\partial {\dot{x}}^2 (x,u,\dot{x})$ is positive definite as a quadratic form;

  \item [\textbf{(L2)}] \textbf{Superlinear Growth}. For every $(x,u)\in M\times\R$, $L(x,u,\dot{x})$ is superlinear with respect to $\dot{x}$, that is,          \begin{equation*}
                        \lim_{|\dot{x}|_{x}\rightarrow\infty}\frac{L(x,u,\dot{x})}{|\dot{x}|_{x}}\rightarrow\infty,
                        \end{equation*}
                        where $\|\cdot\|_{x}$ denotes the norm on $T_{x}M$ induced by the Riemannian metric,
  \item [\textbf{(L3)}] \textbf{Moderate decreasing}. There is constant $\Lambda>0$ such that for every $(x,u,\dot{x})\in TM\times\R$,
                        \begin{equation*}
                        -\Lambda\leq \frac{\partial L}{\partial u}(x,u,\dot{x})<0.
                        \end{equation*}
\end{itemize}
From the first two assumptions on $H$ and $L$, there is a global diffeomorphism $\mathcal{L}:T^{\ast}M\times\R\rightarrow TM\times\R$ defined as
\begin{equation}\label{lt}
\mathcal{L}(x,u,p)=(x,u,\frac{\partial H}{\partial p}(x,u,p)),
\end{equation}
and its inverse is
\begin{equation}
\mathcal{L}^{-1}(x,u,\dot{x})=(x,u,\frac{\partial L}{\partial\dot{x}}(x,u,\dot{x})),
\end{equation}
where $\frac{\partial L}{\partial\dot{x}}$ and $\frac{\partial H}{\partial p}$ denote the partial derivative of $L$ and $H$ with respect to their third arguments respectively. Note that our definition is different from the usual one that reverse roles of $\mathcal{L}$ and $\mathcal{L}^{-1}$. We define the metric $d_{T^{\ast}M\times\R}:=\mathcal{L}^{\ast}d_{TM\times\R}$, since $\mathcal{L}$ is a global diffeomorphism, we notice that by definition, metrics $d_{TM\times\R}$ and $d_{T^{\ast}M\times\R}$ are both induced by Riemannian metrics on their domains.

\subsection{Contact Hamilton ODE}
Let $H:T^{\ast}M\times\mathbb{R}\rightarrow\mathbb{R}$ be a $C^{r}$ contact Hamiltonian, there is an ODE system called contact Hamilton equations associated to $H$, which is defined as
\begin{equation}\label{hjech}
\begin{cases}
\dot{x}=\frac{\partial H}{\partial p},\\
\dot{p}=-\frac{\partial H}{\partial x}-\frac{\partial H}{\partial u}p,\\
\dot{u}=\frac{\partial H}{\partial p}p-H.
\end{cases}
\end{equation}
We shall denote $\Phi^{t}_{H}$ the local flow generated by the above system and call it the contact Hamilton flow, since it is a generalization of the classical Hamilton flow. Recall that along an orbit of classical Hamilton flow, the energy function $H$ preserves. Now for the contact Hamiltonian flow, the variation of energy function along an orbit $\Phi^{\tau}_{H}z:[t_{0},t_{1}]\rightarrow T^{\ast}M\times\mathbb{R}$ of \eqref{hjech} takes the form
\begin{equation}\label{ef}
\frac{d}{d\tau}H(\Phi^{\tau}_{H}z)=-\frac{\partial H}{\partial u}(\Phi^{\tau}_{H}z)\cdot H(\Phi^{\tau}_{H}z).
\end{equation}
The above equation and assumption (H3) imply that for $t>0$,
\begin{equation}\label{eb}
e^{-\Lambda t}|H(z)|\leq|H(\Phi^{t}_{H}z)|<|H(z)|.
\end{equation}

\subsection{Contact Hamilton-Jacobi equations}
This section is devoted to introducing some preliminary results on the contact Hamilton-Jacobi equations \eqref{evo} and \eqref{sta} which are useful for our proof of Theorem \ref{main}. To use the dynamical method to study equations \eqref{evo} and \eqref{sta}, we carry out the following definition, see also \cite{SWY,WWY,WWY1}.
\begin{definition}\label{sg}
Let $L:TM\times\R\rightarrow\R$ be a $C^{r}$ Lagrangian satisfying (L1)-(L3), for each $t\geq0$, there is a functional operator $T_t:C^0(M,\R)\rightarrow C^0(M,\R)$ associated to $L$ implicitly defined as
\begin{equation}\label{uxt}
T_t\varphi(x)=\inf_{\gm(t)=x}\left\{\varphi(\gm(0))+\int_0^tL(\gm(\tau),T_\tau\varphi(\gm(\tau)),\dot{\gm}(\tau))d\tau\right\},
\end{equation}
where the infimum is taken among continuous and piecewise $C^1$ curves $\gm:[0,t]\rightarrow M$.
\end{definition}

Let us collect some elementary properties of the operator $T_{t}$ shown by \cite{SWY,WWY}.
\begin{proposition}\label{prre-1}
Let $T_t:C^0(M,\R)\rightarrow C^0(M,\R)$ be defined as \eqref{uxt}, there hold
\begin{itemize}
  \item [(i).] {\rm Monotonicity}. For every $\varphi,\psi\in C^0(M,\R)$, if $\varphi\leq \psi$, then $T_t\varphi(x)\leq T_t\psi(x)$.

  \item [(ii).] {\rm Non-expansiveness}. For every $\varphi,\psi\in C^0(M,\R)$, $\|T_t\varphi(x)-T_t\psi(x)\|_{C^0}\leq \|\varphi(x)-\psi(x)\|_{C^0}$.

  \item [(iii).] {\rm Variational principle}. The infimum in \eqref{uxt} can be attained by an absolutely continuous curve $\gm_{x,t}:[0,t]\rightarrow M$ with $\gm_{x,t}(t)=x$ called a minimizer of $T_t\varphi(x)$.

  \item [(iv).] {\rm Semigroup Property}. For $s,t\geq0, T^{-}_{s+t}=T^{-}_{s}\circ T^{-}_{t}$.
\end{itemize}
\end{proposition}

The relationship between the operator $T_{t}$ and contact Hamilton-Jacobi equation is described in the following
\begin{proposition}\label{prre-2}
Let $T_t:C^0(M,\R)\rightarrow C^0(M,\R)$ be defined as \eqref{uxt}, there hold
\begin{itemize}
  \item [(i).] {\rm Evolutionary solution}. For every $\varphi\in C^0(M,\R)$, $u(t,x)=T_t\varphi(x)$ is the unique viscosity solution of the equation (\ref{evo}) with the initial condition $u(0,x)=\varphi(x)$. Moreover, the minimizers $\gm_{x,t}$ are $C^{r}$ and if one set \[(x(\tau), u(\tau), p(\tau))=\mathcal{L}^{-1}(\gm_{x,t}(\tau),T_\tau\varphi(\gm_{x,t}(\tau)),\dot{\gm}_{x,t}(\tau)),\] then the curve $(x(\tau), u(\tau), p(\tau))$ is of class $C^{r-1}$ and satisfies the contact Hamilton equations \eqref{hjech} on its domain. Thus we call the operator family $\{T^{-}_{t}\}_{t\geq0}$ the solution semigroup associated to \eqref{evo}.

  \item [(ii).] {\rm Stationary solution}. $u_{-}\in C^{0}(M,\R)$ is a viscosity solution of (\ref{sta}) if and only if for each $t\geq 0$, $T_tu_{-}(x)=u_{-}(x)$, which is also equivalent to
      \begin{itemize}
      \item [(1)] for each continuous and piecewise $C^1$ curve $\gm:[t_0,t_1]\rightarrow M$, we have
      \begin{equation*}
      u_-(\gm(t_1))- u_-(\gm(t_0))\leq\int_{t_0}^{t_1}L(\gm(\tau),u_-(\gm(\tau)),\dot{\gm}(\tau))d\tau;
      \end{equation*}
      \item [(2)] for any $x\in M$, there exists a $C^{r}$ curve $\gm_{x,-}:(-\infty,0]\rightarrow M$ with $\gm_{x,-}(0)=x$ such that for any $t_{0}\leq t_{1}\leq0$, we have
      \begin{equation*}
      u_-(\gm_{x,-}(t_{1}))-u_-(\gm_{x,-}(t_{0}))=\int_{t_{0}}^{t_{1}}L(\gm_{x,-}(\tau),u_-(\gm_{x,-}(\tau)),\dot{\gm}_{x,-}(\tau))d\tau
      \end{equation*}
      and if one set $(x(\tau), u(\tau), p(\tau))=\mathcal{L}^{-1}(\gm_{x,-}(\tau),u_{-}(\gm_{x,-}(\tau)),\dot{\gm}_{x,-}(\tau))$, then the curve $(x(\tau), u(\tau), p(\tau))$ is of class $C^{r-1}$ and satisfies the contact Hamilton equations \eqref{hjech} on its domain.
      \end{itemize}
\end{itemize}
\end{proposition}

\begin{remark}\label{c-semigroup}
For any $c\in\R$, we could similarly define a functional operator $T^{c}_t\varphi:C^0(M,\R)\rightarrow C^0(M,\R)$ by replacing $L$ in \eqref{uxt} with $L+c$. All properties listed in Proposition \ref{prre-1} and \ref{prre-2} hold if one replace $T_{t}, H$ and $L$ with $T^{c}_{t}, H-c$ and $L+c$ respectively. But for different $c$, the asymptotic behavior of $T^{c}_t\varphi(x)$ as $t\rightarrow\infty$ may be essentially different. This is related to the notion of admissible value set, see Appendix B.
\end{remark}

As mentioned in Remark \ref{c-semigroup}, whether $0$ belongs to $\mathcal{C}_{H}$ has essential effects on the asymptotic behavior of $T_{t}$. Precisely, we have
\begin{proposition}\label{prre-3}
Under assumptions (H1)-(H3), if $0\in\mathcal{C}_{H}$, then for a given $\delta>0$, $\{T_t\varphi(x)\}_{t\geq\delta}$ are uniformly bounded and equi-Lipschitz with respect to $t$. Moreover, $T_t\varphi(x)$ converges to a Lipschitz function ${u}_-(x)$  as $t$ goes to $+\infty$ and ${u}_-(x)$ is  a viscosity solution of the stationary equation (\ref{sta}).
\end{proposition}

\begin{remark}
By \cite{SWY}, for any $c\in \mathcal{C}_H$,  the properties listed in Proposition \ref{prre-2} hold if one replace $T_{t}$ and $H$ with $T^c_t$ and $H-c$ respectively. Without loss of generality, we always assume $0\in\mathcal{C}_H$ and our results also apply to any $c\in\mathcal{C}_H$.
\end{remark}

\section{A key lemma}
In this section, we shall formulate a lemma which will be useful in our proof of the main theorem. We would like to mention that it is a complete version of a result obtained in \cite[Lemma 6.5]{SWY}. Let $H:T^{\ast}M\times\R\rightarrow\R$ be a $C^{r},r\geq3$ Hamiltonian satisfying (H1)-(H3), $L$ is the Lagrangian associated to $H$. For a fixed semiconcave function $u$, we define
\begin{itemize}
  \item a $C^{0}$ Lagrangian $l_{u}:TM\rightarrow\mathbb{R}$ by
        \begin{equation}\label{ml}
        l_{u}(x,\dot{x}):=L(x,u(x),\dot{x})-\partial u(x,\dot{x})
        \end{equation}

  \item a homeomorphism $\mathcal{L}_{u}:T^{\ast}M\rightarrow TM$ by
        \begin{equation}\label{gen-leg}
        \mathcal{L}_{u}(x,p):=\pi\mathcal{L}(x,u(x),p)=(x,\frac{\partial H}{\partial p}(x,u(x),p)),
        \end{equation}
        and its inverse $\mathcal{L}^{-1}_{u}(x,\dot{x})=(x,\frac{\partial L}{\partial\dot{x}}(x,u(x),\dot{x}))$.
\end{itemize}
and the following theorem holds:
\begin{lemma}\label{key}
Given $\beta>0$, there are positive constants $\alpha,\Delta$ only depending on $L,u$ and $\beta$ such that
\begin{itemize}
  \item if $d_{TM}((x,\dot{x}),\mathcal{L}_{u}\bar{\mathcal{G}}_{u})\leq\Delta$, then
        \begin{equation*}
        l_{u}(x, \dot{x})\geq\alpha\cdot d_{TM}((x,\dot{x}),\mathcal{L}_{u}\bar{\mathcal{G}}_{u})^{2}-\min_{p\in D^{*}u(x)}H(x,u(x),p)
        \end{equation*}

  \item if $d_{TM}((x,\dot{x}),\mathcal{L}_{u}\bar{\mathcal{G}}_{u})>\Delta$, then
        \begin{equation*}
        l_{u}(x, \dot{x})\geq\beta-\min_{p\in D^{*}u(x)}H(x,u(x),p).
        \end{equation*}
\end{itemize}
\end{lemma}

\begin{proof}
For $x\in\mathcal{D}_{u}$, we define
\begin{equation*}
F(x,\dot{x}):=L(x,u(x),\dot{x})-\langle d_{x}u(x),\dot{x}\rangle+H(x,u(x),d_{x}u(x)).
\end{equation*}
By convexity and duality of $L$ and $H$, $F(x,\dot{x})\geq0$. By definition, $F$ is smooth on each tangent space $T_{x}M,x\in\mathcal{D}_{u}$ and we have
\begin{equation}\label{derivative}
\frac{\partial F}{\partial\dot{x}}=\frac{\partial
L}{\partial\dot{x}}(x,u(x),\dot{x})-d_{x}u(x),\,\,\,\,\frac{\partial^{2}F}{\partial\dot{x}^{2}}=\frac{\partial^{2}L}{\partial\dot{x}^{2}}(x,u(x),\dot{x}),
\end{equation}
this implies that $F$ is strictly convex with respect to $\dot{x}$ and $F(x,\dot{x})=0$ if and only if $(x,\dot{x})=\mathcal{L}_{u}(x,d_{x}u(x))$ or $\dot{x}=\frac{\partial H}{\partial p}(x,u(x),d_{x}u(x))$.

Since $u$ is Lipschitz continuous on $M$, there exists positive constant $B$ such that if $x\in\mathcal{D}_{u}$,
\begin{equation*}
|d_{x}u(x)|_{x}\leq B_{0},\,\,\,|H(x,u(x),d_{x}u(x))|\leq B_{0}.
\end{equation*}
Thus by the compactness of $M$ and (L2)-(L3), for a given $\beta>0$ and any $x\in\mathcal{D}_{u}$, there exists constant $\Delta_{0}:=\Delta_{0}(L,u,\beta)>0$ such that
\begin{equation}\label{ieq:3}
F(x,\dot{x})\geq L(x,u(x),\dot{x})-B(|\dot{x}|_{x}+1)\geq\beta\,\,\,\,\text{ if }\,\,\|\dot{x}\|_{x}\geq\Delta_{0}.
\end{equation}
Since $\mathcal{L}_{u}\bar{\mathcal{G}}_{u}$ and $\{(x,\dot{x}):|\dot{x}|_{x}\leq\Delta_{0}\}$ are compact subsets of $TM$, we could choose $\Delta:=\Delta(L,u,\beta)>0$ large enough such that
$$
\{(x,\dot{x}):|\dot{x}|_{x}\leq\Delta_{0}\}\subseteq\{(x,\dot{x}):d_{TM}((x,\dot{x}),\mathcal{L}_{u}\bar{\mathcal{G}}_{u})\leq\Delta-1\}.
$$
Thus if $x\in\mathcal{D}_{u},d_{TM}((x,\dot{x}),\mathcal{L}_{u}\bar{\mathcal{G}}_{u})>\Delta$, we have $|\dot{x}|_{x}\geq\Delta_{0}$ and by definition,
\begin{equation}\label{3}
l_{u}(x,\dot{x})=F(x,\dot{x})-H(x,u(x),\partial_{x}u(x))\geq\beta-H(x,u(x),d_{x}u(x)).
\end{equation}

Again we notice that $|\frac{\partial H}{\partial p}(x,u(x),d_{x}u(x))|_{x}$ is bounded above and the set  \[\{(x,\dot{x}):d_{TM}((x,\dot{x}),\mathcal{L}_{u}\bar{\mathcal{G}}_{u})\leq\Delta\}\] is a compact subset of $TM$, there exists $\Delta_{1}:=\Delta_{1}(L,u,\beta)>0$ such that
\begin{equation*}
\{(x,\dot{x}):d_{TM}((x,\dot{x}),\mathcal{L}_{u}\bar{\mathcal{G}}_{u})\leq\Delta+1\}\subseteq\mathcal{S}:=\{(x,\dot{x}):|\dot{x}-\frac{\partial H}{\partial p}(x,u(x),d_{x}u(x))|_{x}\leq\Delta_{1}\}.
\end{equation*}
By (L1) and (\ref{derivative}), for any $(x,\dot{x})\in \mathcal{S}$, there exists $\alpha>0$ such that
$$
\frac{\partial^2F}{\partial\dot{x}^2}\geq2\alpha\cdot Id
$$
with respect to $|\cdot|_{x}$ on each tangent space $T_{x}M$. Hence by \eqref{dtm-1}, it follows that for $x\in\mathcal{D}_{u}$ and $(x,\dot{x})\in \mathcal{S}$, there holds
\begin{align*}
F(x,\dot{x})\geq&\alpha\cdot|\dot{x}-\frac{\partial H}{\partial p}(x,u(x),d_x u(x))|_{x}^2\geq\alpha\cdot d_{TM}((x,\dot{x}),\mathcal{L}_{u}\bar{\mathcal{G}}_{u})^{2},
\end{align*}which implies, from the definition of $l_{u}$,
\begin{align*}
l_{u}(x,\dot{x})\geq&\alpha\cdot d_{TM}((x,\dot{x}),\mathcal{L}_{u}\bar{\mathcal{G}}_{u})^{2}-H(x,u(x),d_{x}u(x)).
\end{align*}

Finally, for any $(x,\dot{x})\in TM$ with $d_{TM}((x,\dot{x}),\mathcal{L}_{u}\bar{\mathcal{G}}_{u})\leq\Delta$, the compactness of $D^{\ast}u(x)$ implies that there exists $p_{0}\in D^{*}u(x)$ such that
$$
H(x,u(x),p_{0})=\min_{p\in D^{*}u(x)}H(x,u(x),p).
$$
We choose a sequence $(x_{k},\dot{x}_{k})\rightarrow(x,\dot{x})$ such that $x_{k}\in\mathcal{D}_{u}$ and $d_{x}u(x_{k})\rightarrow p_{0}$, then
\begin{align*}
l_{u}(x,\dot{x})=&\max_{p\in D^{*}u(x)}\{L(x,u(x),\dot{x})-\langle p,\dot{x}\rangle\}\\
\geq&L(x,u(x),\dot{x})-\langle p_{0},\dot{x}\rangle\\
=&\lim_{k\rightarrow\infty}l_{u}(x_{k},\dot{x}_{k})\\
\geq&\lim_{k\rightarrow\infty}\alpha\cdot d_{TM}((x_{k},\dot{x}_{k}),\mathcal{L}_{u}\bar{\mathcal{G}}_{u})^{2}-H(x_{k},u(x_{k}),d_{x}u(x_{k}))\\
=&\alpha\cdot d_{TM}((x,\dot{x}),\mathcal{L}_{u}\bar{\mathcal{G}}_{u})^{2}-H(x_{k},u(x_{k}),p_{0})\\
=&\alpha\cdot d_{TM}((x,\dot{x}),\mathcal{L}_{u}\bar{\mathcal{G}}_{u})^{2}-\min_{p\in D^{*}u(x)}H(x,u(x),p),
\end{align*}
where the second inequality holds since for $k$ sufficiently large, $(x_k,\dot{x}_k)\in \mathcal{S}$. For the case with  \[d_{TM}((x,\dot{x}),\mathcal{L}_{u}\bar{\mathcal{G}}_{u})\geq\Delta,\] by the same argument above, we have
\begin{equation*}
l_{u}(x,\dot{x})=\lim_{k\rightarrow\infty}l_{u}(x_{k},\dot{x}_{k})
\geq\lim_{k\rightarrow\infty}\beta-H(x_{k},u(x_{k}),d_{x}u(x_{k}))
=\beta-\min_{p\in D^{*}u(x)}H(x,u(x),p),
\end{equation*}
where the inequality holds since for $k$ sufficiently large, $d_{TM}((x_{k},\dot{x}_{k}),\mathcal{L}_{u}^{-1}\bar{\mathcal{G}}_{u})>\Delta-1$, thus $|\dot{x}_{k}|_{x_{k}}\geq\Delta_{0}$. This completes the proof.\end{proof}

\section{Proof of the main theorem}
This section is devoted to a proof of  Theorem \ref{main}, we will always assume that $H=H(x,u,p)$ is a $C^{r},r\geq2$ contact Hamiltonian satisfying (H1)-(H3) and $L$ is the Lagrangian associated to $H$. We use $u_{-}$ to denote the viscosity solution of (\ref{sta}) and the uniqueness of $u_{-}$ is settled in Appendix A. Thus to prove Theorem \ref{main}, we only need to show that the inequalities \eqref{ecs} and \eqref{ecj} hold.

\subsection{Step I}
We shall give definition of $\lambda(H)$ and make some constructions in this step. Let $I$ be a compact interval and $B>0$, by (H1)-(H3), the set
\begin{equation}\label{cps}
K_{I,B}:=\{(x,u,p)|u\in I,|H(x,u,p)|\leq B\}
\end{equation}
is a compact subset of $T^{\ast}M\times\R$. Before getting into the proof, we shall prove a lemma which is used in the following construction.
\begin{lemma}\label{com-c}
There is a compact subset $K_{0}$ of $T^{\ast}M\times\R$ only depending on $H$ such that for any $\varphi\in C^{0}(M,\R)$, there is $t(\varphi)>0$ satisfying
\begin{equation}
\cup_{t\geq t(\varphi)}\bar{\mathcal{J}}^{1}_{T_{t}\varphi}\cup\bar{\mathcal{J}}^{1}_{u_{-}}\subseteq K_{0}.
\end{equation}
\end{lemma}

\begin{proof}
Proposition \ref{prre-3} implies that for any $\varphi\in C^{0}(M,\R)$, there is $t(\varphi)>0$ such that for $t\geq t(\varphi)$, \begin{equation}\label{def-t}
\|T_{t}\varphi-u_{-}\|_{C^{0}}\leq1.
\end{equation}
Thus by uniqueness of $u_{-}$, for any $t\geq t(\varphi)$, $\|T_{t}\varphi\|_{C^{0}}\leq\|u_{-}\|_{C^{0}}+1=a:=a(H)$. By the assumption (L2), there is $b:=b(H)>0$ such that $L(x,a,\dot{x})\geq|\dot{x}|_{x}-b$. It follows that for $t\geq t(\varphi)+1$ and any minimizer $\gm_{x,t}$,
\begin{align*}
&T_{t}\varphi(\gm_{x,t}(t))-T_{t-1}\varphi(\gm_{x,t}(t-1))\\
=&\int^{t}_{t-1} L(\gm_{x,t}(\tau),T_{\tau}\varphi(\gm_{x,t}(\tau),\tau),\dot{\gm}_{x,t}(\tau))d\tau\\
\geq&\int^{t}_{t-1}L(\gm_{x,t}(\tau),a,\dot{\gm}_{x,t}(\tau))d\tau\\
\geq&\int^{t}_{t-1}|\dot{\gm}_{x,t}(\tau)|_{\gm_{x,t}(\tau)}d\tau-b,\\
\end{align*}
where the second inequality is owing to the monotonicity assumption. However, by definition of $t(\varphi)$, we have
\begin{equation*}
|T_{t}\varphi(\gm_{x,t}(t))-T_{t-1}\varphi(\gm_{x,t}(t-1))|\leq 2a.
\end{equation*}
Hence, one can find $t_{0}\in[t-1,t]$ such that $|\dot{\gm}(t_0)|\leq2a+b$, this implies that there is $B:=B(H)$, for $z_{0}=\mathcal{L}^{-1}(\gm_{x,t}(t_{0}),T_{t_0}\varphi(\gm_{x,t}(t_{0})),\dot{\gm}_{x,t}(t_{0}))$, $|H(z_{0})|\leq B$. Thus by \eqref{eb}, $|H(x,T_{t}\varphi(x),p)|=|H(\Phi^{t-t_{0}}_{H}z_{0})|<|H(z_{0})|\leq B$. Let $I=[-a,a]$, we deduce that for any $t\geq t(\varphi)$ and any minimizer $\gm_{x,t}$ of $T_{t}\varphi(x)$, \begin{equation}
\mathcal{L}^{-1}(x,T_{t}\varphi(x),\dot{\gm}_{x,t}(t))\in\{(x,u,p)|u\in I,|H(x,u,p)|\leq B\}:=K_{0}(H)
\end{equation}
is a compact subset of $T^{\ast}M\times\R$ only depending on $H$.
\end{proof}

We make the following definitions which are used in the next three steps:
\begin{equation}\label{hbaa}
\begin{split}
&\lambda(H):=\inf\{\frac{\partial H}{\partial u}(x,u,p)|(x,u,p)\in K_{0}\}>0,\\
&\bar{H}(x,u,p):=\lambda(u-u_{-}(x))+H(x,u_{-}(x),p),\\
&\bar{L}(x,u,\dot{x}):=\lambda(u_{-}(x)-u)+L(x,u_{-}(x),\dot{x}).
\end{split}
\end{equation}
We denote by $T_{t},\bar{T}_{t}$ the solution semigroups associated to $L$ and $\bar{L}$ respectively.

\subsection{Step II}
In this step, we prove \eqref{ecs}. For a given $\varphi\in C^{0}(M,\mathbb{R})$, let \[c:=\max\{\|\varphi\|_{C^{0}},\|u_{-}\|_{C^{0}}\}.\] In this step, we shall regard $\pm c$ as constant functions on $M$ and use the notation $T_{t}[\pm c]$ to denote their images under the functional operator $T_{t}$. On one hand, by Proposition \ref{prre-1} and \ref{prre-2}, it is clear that
\begin{align}\label{cp1}
T_{t}[-c]\leq& T_{t}\varphi\leq T_{t}[c]\\
T_{t}[-c]\leq& u_{-}\leq T_{t}[c].
\end{align}
On the other hand, we have
\begin{lemma}
Let $t_{c}=\max\{t(c),t(-c)\}>0$ where $t(\cdot)$ is defined in \eqref{def-t}, then for $t\geq0$,
\begin{align}\label{cp2}
\bar{T}_{t}\circ T_{t_{c}}[-c]\leq& T_{t+t_{c}}[-c],\quad T_{t+t_{c}}[c]\leq\bar{T}_{t}\circ T_{t_{c}}[c].
\end{align}
\end{lemma}

\Proof We shall only prove the first inequality, the proof of second one is the same. We shall argue by contradiction, let $\psi=T_{t_{c}}[-c]$ and suppose that there is $(x_{1},t_{1})\in M\times(0,+\infty)$ such that
\begin{equation}\label{contra}
T_{t_{1}}\psi(x_{1})<\bar{T}_{t_{1}}\psi(x_{1}).
\end{equation}

Let $\gm_{x_{1},t_{1}}:[0,t_{1}]\rightarrow M$ be the $C^{r}$ curve that achieves the infimum for $T_{t_{1}}\psi(x_{1})$. Note that $T_0\psi(x)=\bar{T}_0\psi(x)=\psi(x)$ for each $x\in M$, there is $t_{0}\in[0,t_{1})$ such that for $\tau\in[t_{0},t_{1}]$,
\begin{equation}\label{4}
\begin{split}
&T_{t_{0}}\psi(\gm_{x_{1},t_{1}}(t_{0}))=\bar{T}_{t_{0}}\psi(\gm_{x_{1},t_{1}}(t_{0})),\\
&T_{\tau}\psi(\gm_{x_{1},t_{1}}(\tau))<\bar{T}_{\tau}\psi(\gm_{x_{1},t_{1}}(\tau)).
\end{split}
\end{equation}
Since $\gm_{x_1,t_1}:[0,t_{1}]\rightarrow M$ is part of the minimizer of $T_{t_1+t_{c}}[-c](x_1)$, according to Lemma \ref{com-c}, there exists a compact subset $\mathcal{L}K_{0}$ of $TM\times\R$ such that for $\tau\in[t_{0},t_{1}]$,
\begin{equation}\label{cp4}
\begin{split}
(\gm_{x_1,t_1}(\tau),T_{\tau}\psi(\gm_{x_{1},t_{1}}(\tau)),\dot{\gm}_{x_1,t_1}(\tau))\subset\mathcal{L}K_{0},\\
(\gm_{x_1,t_1}(\tau),u_{-}(\gm_{x_{1},t_{1}}(\tau)),\dot{\gm}_{x_1,t_1}(\tau))\subset\mathcal{L}K_{0}
\end{split}
\end{equation}
By non-expansiveness of the operator $\bar{T}_{t}$, one can find that for $\tau\in[t_{0},t_{1}]$,
\begin{equation*}
\|\bar{T}_{\tau}\psi-u_{-}\|_{C^{0}}\leq\|\psi-u_{-}\|_{C^{0}}\leq1,
\end{equation*}
thus we have
\begin{equation}\label{cp5}
(\gm_{x_1,t_1}(\tau),\bar{T}_{\tau}\psi(\gm_{x_{1},t_{1}}(\tau)),\dot{\gm}_{x_1,t_1}(\tau))\subset\mathcal{L}K_{0}.
\end{equation}
By the definition of solution semigroup, we have
\begin{align*}
\begin{split}
T_{t_{1}}\psi(x_{1})=T_{t_{0}}\psi(\gm_{x_{1},t_{1}}(t_{0}))+\int_{t_{0}}^{t_{1}}L(\gm_{x_{1},t_{1}}(\tau),T_{\tau}\psi(\gm_{x_{1},t_{1}}(\tau)),\dot{\gm}_{x_{1},t_{1}}(\tau))d\tau,\\
\bar{T}_{t_{1}}\psi(x_{1})\leq\bar{T}_{t_{0}}\psi(\gm_{x_{1},t_{1}}(t_{0}))+\int_{t_{0}}^{t_{1}}\bar{L}(\gm_{x_{1},t_{1}}(\tau),\bar{T}_{\tau}\psi(\gm_{x_{1},t_{1}}(\tau)),\dot{\gm}_{x_{1},t_{1}}(\tau))d\tau.
\end{split}
\end{align*}

From the definition of $\bar{L}$, it follows that $\frac{\partial \bar{L}}{\partial u}\equiv-\lambda$ and $\bar{L}(x,u_-(x),\dot{x})=L(x,u_{-}(x),\dot{x})$. So from the above two equations, we could calculate as
\begin{align*}
&\,\,\,\,T_{t_{1}}\psi(x_{1})-\bar{T}_{t_{1}}\psi(x_{1})\\
\geq&\int_{t_{0}}^{t_{1}}L(\gm_{x_{1},t_{1}}(\tau),T_{\tau}\psi(\gm_{x_{1},t_{1}}(\tau)),\dot{\gm}_{x_{1},t_{1}}(\tau))-\bar{L}(\gm_{x_{1},t_{1}}(\tau),\bar{T}_{\tau}\psi(\gm_{x_{1},t_{1}}(\tau)),\dot{\gm}_{x_{1},t_{1}}(\tau))d\tau\\
\geq&\int_{t_{0}}^{t_{1}}\lambda\bigg[\bar{T}_{\tau}\psi(\gm_{x_{1},t_{1}}(\tau))-u_{-}(\gm_{x_{1},t_{1}}(\tau))\bigg]-\lambda\bigg[T_{\tau}\psi(\gm_{x_{1},t_{1}}(\tau))-u_{-}(\gm_{x_{1},t_{1}}(\tau))\bigg]d\tau\\
=&\int_{t_{0}}^{t_{1}}\lambda\bigg[\bar{T}_{\tau}\psi(\gm_{x_{1},t_{1}}(\tau))-T_{\tau}\psi(\gm_{x_{1},t_{1}}(\tau))\bigg]d\tau>0,\\
\end{align*}
where the second inequality uses the equations \eqref{cp4}, \eqref{cp5} and the last one uses by \eqref{4}. This contradicts to our assumption \eqref{contra}.
\end{proof}

Combining \eqref{cp1}-\eqref{cp2}, we get for $t\geq0$
\begin{equation}
\begin{split}
&\bar{T}_{t}\circ T_{t_{c}}[-c]\leq u_{-}\leq\bar{T}_{t}\circ T_{t_{c}}[c],\\
&\bar{T}_{t}\circ T_{t_{c}}[-c]\leq T_{t+t_{c}}\varphi\leq\bar{T}_{t}\circ T_{t_{c}}[c],
\end{split}
\end{equation}
which implies that
\begin{equation}\label{cp3}
\|T_{t+t_{c}}\varphi-u_{-}\|_{C^{0}}\leq\max\{\|\bar{T}_{t}\circ T_{t_{c}}[c]-u_{-}\|_{C^{0}},\|\bar{T}_{t}\circ T_{t_{c}}[-c]-u_{-}\|_{C^{0}}\}.
\end{equation}
From Proposition \ref{prre-2}, $\bar{T}_{t}\circ T_{t_{c}}[\pm c](x)$ are solutions of the pair of Cauchy problems
\begin{equation}\label{uld}
\begin{cases}
\partial_tu+\bar{H}(x,u,\partial_{x}u)=0,\hspace{0.3cm}(x,t)\in M\times[0,+\infty)\\
u(0,x)=T_{t_{c}}[\pm c].
\end{cases}
\end{equation}
respectively. Then we apply Lemma \ref{dsep} to conclude that for $t\geq0$,
\begin{equation*}
\|\bar{T}_{t}\circ T_{t_{c}}[\pm c]-u_{-}\|_{C^{0}}\leq O(e^{-\lambda t}),
\end{equation*}
Combining this and \eqref{cp3}, we have proved the inequality \eqref{ecs}.

\subsection{Step III}
By applying Lemma \ref{key}, we show the exponential convergence of the pseudo graph of $T_{t}\varphi$. Without loss of generality, let us assume $t>t_{c}$, by Proposition \ref{prre-2} and Lemma \ref{cps},
\[\mathcal{L}^{-1}(\gamma_{x,t},T_{t}\varphi(\gamma_{x,t}),\dot{\gamma}_{x,t}) \quad {\text{and}} \quad \mathcal{L}^{-1}(\gamma_{x,-},u_{-}(\gamma_{x,-}),\dot{\gamma}_{x,-})\]
are $C^{r-1}$ orbits of contact Hamilton flow $\Phi_{H}^{t}$ which are contained in a compact set $\mathcal{L}K_{0}$ only depending on $H$. By the contact Hamilton equations \eqref{hjech} and the boundedness of the norm of $d\mathcal{L}$ on $K_{0}$, there exists $A:=A(H)>\Delta$ such that
\begin{equation}\label{vb}
\begin{split}
&|(\dot{\gamma}_{x,t}(\tau),\ddot{\gamma}_{x,t}(\tau))|_{(\gamma_{x,t}(\tau),\dot{\gamma}_{x,t}(\tau))}\leq A,\hspace{0.6cm}\tau\in[1,t]\\
&|(\dot{\gamma}_{x,-}(\tau),\ddot{\gamma}_{x,-}(\tau))|_{(\gamma_{x,-}(\tau),\dot{\gamma}_{x,-}(\tau))}\leq A,\,\,\,\tau\in(-\infty,0],
\end{split}
\end{equation}
where $\Delta$ is given by Lemma \ref{key}.

From Section 2 or \cite{CS}, we know that $u_{-}$ is locally semiconcave on $M$. Since $u_{-}$ is a solution to equation \eqref{sta}, for any $x\in M$ and $p\in D^{\ast}u_{-}(x)$
\begin{equation}\label{hje-0}
H(x, u_{-}(x),p)=0.
\end{equation}
An estimate of contact Hamiltonian $H$ on the 1-graph of $T_{t}\varphi$ could also be given by the following
\begin{lemma}\label{eed}
There exists $\tilde{B}:=\tilde{B}(\varphi,H)>0$ such that for $t>t_{c}$,
\begin{equation}
\max_{p\in D^{\ast}T_{t}\varphi(x)}|H(x,T_{t}\varphi(x),p)|\leq\tilde{B}e^{-\lambda t}.
\end{equation}
\end{lemma}

\begin{proof}
For any $x\in M$ and $p\in D^{\ast}T_{t}\varphi(x)$, there exists a minimizer $\gamma_{x,t}$ of $T_{t}\varphi(x)$ such that $(x,T_{t}\varphi(x),p)=
\mathcal{L}^{-1}(x,T_{t}\varphi(x),\dot{\gamma}_{x,t}(t))$. By Proposition \ref{prre-2}, for $\tau\in[0,t]$, $z_{x,t}(\tau):=\mathcal{L}^{-1}(\gamma_{x,t}(\tau),T_{t}\varphi(\gamma_{x,t}(\tau)),\dot{\gamma}_{x,t}(\tau))$ lies on an orbit of the contact Hamiltonian flow $\Phi_{H}^{t}$. By Lemma \ref{cps}, we have that $z_{x,t}|_{[t_{c},t]}\in K_{0}$ and for $\tau\geq t_{c}$, $\frac{\partial H}{\partial u}(z_{x,t}(\tau))\geq\lambda>0$. Then equations \eqref{ef} and \eqref{eb} imply that
\begin{equation}
|H(z_{x,t}(t))|\leq|H(z_{x,t}(t_{c}))|e^{\lambda(t_{c}-t)}\leq Be^{\lambda(t_{c}-t)},
\end{equation}
where $B$ is defined in the proof of Lemma \ref{cps}. Now we choose $\tilde{B}(\varphi,H)=e^{\lambda t_{c}}B$ to complete the proof.
\end{proof}

We are ready for applying Lemma \ref{key} to prove \eqref{ecj}. Let us first take $u=u_{-}$ and there are two cases to dealt with, namely,
\begin{itemize}
  \item $d_{TM}((x,\dot{\gamma}_{x,t}(t)),\mathcal{L}_{u_{-}}\bar{\mathcal{G}}_{u_{-}})\geq 2A$,
  \item $d_{TM}((x,\dot{\gamma}_{x,t}(t)),\mathcal{L}_{u_{-}}\bar{\mathcal{G}}_{u_{-}})<2A$.
\end{itemize}

If $d_{TM}((x,\dot{\gamma}_{x,t}(t)),\mathcal{L}_{u_{-}}\bar{\mathcal{G}}_{u_{-}})\geq2A$, note that $A>\Delta$, we integral $l_{u_{-}}$ along $\gamma_{x,t}$ over the interval $[t-1,t]$ and obtain
\begin{align*}
&\int^{t}_{t-1}l_{u_{-}}(\gamma_{x,t}(\tau),\dot{\gamma}_{x,t}(\tau))d\tau+[u_{-}(\gamma_{x,t}(t))-u_{-}(\gamma_{x,t}(t-1))]\\
=&\int^{t}_{t-1}L(\gamma_{x,t}(\tau),u_{-}(\gamma_{x,t}(\tau)),\dot{\gamma}_{x,t}(\tau))d\tau\\
\leq&\int^{t}_{t-1}L(\gamma_{x,t}(\tau),T_{\tau}\varphi(\gamma_{x,t}(\tau)),\dot{\gamma}_{x,t}(\tau))d\tau+\Lambda\int^{t}_{t-1}|T_{\tau}\varphi(\gamma_{x,t}(\tau))-u_{-}(\gamma_{x,t}(\tau))|d\tau\\
=&[T_{t}\varphi(x)-T_{t-1}\varphi(\gamma_{x,t}(t-1))]+\Lambda\int^{t}_{t-1}|T_{\tau}\varphi(\gamma_{x,t}(\tau))-u_{-}(\gamma_{x,t}(\tau))|d\tau\\
\leq&[T_{t}\varphi(x)-T_{t-1}\varphi(\gamma_{x,t}(t-1))]+O(e^{-\lambda t}),\\
\end{align*}
where for the first equality we use the definition of $l_{u}$ and Lemma \ref{dd-r}, for the last inequality we use Lemma \ref{dsep}. Again by Lemma \ref{dsep}, we have
\begin{equation}\label{5}
\int^{t}_{t-1}l_{u_{-}}(\gamma_{x,t}(\tau),\dot{\gamma}_{x,t}(\tau))d\tau\leq O(e^{-\lambda t})
\end{equation}
However, by Lemma \ref{key}, $l_{u_{-}}(\gamma_{x,t}(\tau),\dot{\gamma}_{x,t}(\tau))\geq\beta$, thus
\begin{equation*}
\int^{t}_{t-1}l_{u_{-}}(\gamma_{x,t}(\tau),\dot{\gamma}_{x,t}(\tau))d\tau\geq\beta>0,
\end{equation*}
this contradicts to \eqref{5} as $t\rightarrow\infty$.

If $d_{TM}((x,\dot{\gamma}_{x,t}(t)),\mathcal{L}_{u_{-}}\bar{\mathcal{G}}_{u_{-}})\leq2A$, let $\delta_0:=d_{TM}((x,\dot{\gamma}_{x,t}(t)),\mathcal{L}_{u_{-}}\bar{\mathcal{G}}_{u_{-}})$. Then we notice that by \eqref{vb}, on the interval $[t-\frac{\delta_0}{2A},t]$,
\begin{equation}
d_{TM}((\gamma_{x,t}(\tau),\dot{\gamma}_{x,t}(\tau)),\mathcal{L}_{u_{-}}\bar{\mathcal{G}}_{u_{-}})\geq\frac{\delta_0}{2}.
\end{equation}
We integral $l_{u_{-}}$ along $\gamma_{x,t}$ over the interval $[t-\frac{\delta_0}{2A},t]$, by similar calculation as before we obtain \begin{equation}\label{6}
\int^{t}_{t-\frac{\delta_0}{2A}}l_{u_{-}}(\gamma_{x,t}(\tau),\dot{\gamma}_{x,t}(\tau))d\tau\leq O(e^{-\lambda t}).
\end{equation}
Again by Lemma \ref{key}, we could assume $\alpha\leq\frac{\beta}{\Delta^{2}}$ and obtain
\begin{equation}\label{7}
\int^{t}_{t-\frac{\delta_{0}}{2A}}l_{u_{-}}(\gamma_{x,t}(\tau),\dot{\gamma}_{x,t}(\tau))d\tau\geq\min\{\frac{\beta\delta_0}{2A},\frac{\alpha\delta_0^{3}}{8A}\}.
\end{equation}
By \eqref{6} and \eqref{7}, $d_{TM}((x,\dot{\gamma}_{x,t}(t)),\mathcal{L}_{u_{-}}\bar{\mathcal{G}}_{u_{-}})=\delta_0\leq O(e^{-\frac{\lambda}{3}t})$.

For the other side, we take $u=T_{t}\varphi$ and dealt with also two cases
\begin{itemize}
  \item $d_{TM}((x,\dot{\gamma}_{x,-}(0)),\mathcal{L}_{T_{t}\varphi}\bar{\mathcal{G}}_{T_{t}\varphi})\geq2A$,
  \item $d_{TM}((x,\dot{\gamma}_{x,-}(0)),\mathcal{L}_{T_{t}\varphi}\bar{\mathcal{G}}_{T_{t}\varphi})<2A$.
\end{itemize}

If $d_{TM}((x,\dot{\gamma}_{x,-}(0)),\mathcal{L}_{T_{t}\varphi}\bar{\mathcal{G}}_{T_{t}\varphi})\geq2A$, we integral $l_{T_{t}\varphi}$ along $\gamma_{x,-}$ over the time interval $[-1,0]$ and by similar calculation, we obtain
\begin{equation}\label{8}
\int^{0}_{-1}l_{T_{t}\varphi}(\gamma_{x,-}(\tau),\dot{\gamma}_{x,-}(\tau))d\tau\leq O(e^{-\lambda t})
\end{equation}
On the other hand, combining Lemma \ref{key} and Lemma \ref{eed}, we have for $\tau\in[-1,0]$
\begin{equation}
l_{T_{t}\varphi}(\gamma_{x,-}(\tau)),\dot{\gamma}_{x,-}(\tau))\geq\beta-O(e^{-\lambda t}),
\end{equation}
so we take integration and obtain
\begin{equation}
\int^{0}_{-1}l_{T_{t}\varphi}(\gamma_{x,-}(\tau)),\dot{\gamma}_{x,-}(\tau))d\tau\geq\beta-O(e^{-\lambda t}),
\end{equation}
this contradicts to \eqref{8} as $t\rightarrow\infty$.

If $d_{TM}((x,\dot{\gamma}_{x,-}(0)),\mathcal{L}_{T_{t}\varphi}\bar{\mathcal{G}}_{T_{t}\varphi})\leq2A$, let $\delta_1:=d_{TM}((x,\dot{\gamma}_{x,-}(0)),\mathcal{L}_{T_{t}\varphi}\bar{\mathcal{G}}_{T_{t}\varphi})$. Then we notice that by \eqref{vb}, on the interval $[-\frac{\delta_1}{2A},0]$,
\begin{equation}
d_{TM}((\gamma_{x,-}(\tau),\dot{\gamma}_{x,-}(\tau)),\mathcal{L}_{T_{t}\varphi}\bar{\mathcal{G}}_{T_{t}\varphi})\geq\frac{\delta_1}{2}.
\end{equation}
We integral $l_{T_{t}\varphi}$ along $\gamma_{x,-}$ over the interval $[-\frac{\delta_1}{2A},0]$, by similar calculation as before we obtain
\begin{equation}\label{9}
\int^{0}_{-\frac{\delta_1}{2A}}l_{T_{t}\varphi}(\gamma_{x,-}(\tau),\dot{\gamma}_{x,-}(\tau))d\tau\leq O(e^{-\lambda t}).
\end{equation}
As before, we could assume $\alpha\leq\frac{\beta}{\Delta^{2}}$ and obtain
\begin{equation}\label{10}
\int^{0}_{-\frac{\delta_{1}}{2A}}l_{T_{t}\varphi}(\gamma_{x,t}(\tau),\dot{\gamma}_{x,t}(\tau))d\tau\geq\min\{\frac{\beta\delta_1}{2A},\frac{\alpha\delta_1^{3}}{8A}\}-\delta_{1}\cdot O(e^{-\lambda t}).
\end{equation}
By \eqref{9},\eqref{10} and the fact that $\delta_{1}\leq2A$, we obtain $d_{TM}((x,\dot{\gamma}_{x,-}(0)),\mathcal{L}_{T_{t}\varphi}\bar{\mathcal{G}}_{T_{t}\varphi})=\delta_1\leq O(e^{-\frac{\lambda}{3}t})$.

\subsection{Step IV}
We complete the proof of \eqref{ecj} in this last step. It is clear that
\begin{equation}\label{3-1}
\begin{split}
&\{(x,u_{-}(x),\dot{\gamma}_{x,-}(0)):x\in M,\gamma_{x,-} \text{ is a minimizer of }u_{-}(x)\}=\mathcal{L}\bar{\mathcal{J}}^{1}_{u_{-}},\\
&\{(x,T_{t}\varphi(x),\dot{\gamma}_{x,t}(t)):x\in M,\gamma_{x,t} \text{ is a minimizer of }T_{t}\varphi(x)\}=\mathcal{L}\bar{\mathcal{J}}^{1}_{T_{t}\varphi},
\end{split}
\end{equation}
We act $\pi$ on two sides of \eqref{3-1} to obtain
\begin{equation}\label{3-2}
\begin{split}
&\{(x,\dot{\gamma}_{x,-}(0)):x\in M,\gamma_{x,-} \text{ is a minimizer of }u_{-}(x)\}=\pi\mathcal{L}\bar{\mathcal{J}}^{1}_{u_{-}}=\mathcal{L}_{u_{-}}\bar{\mathcal{G}}_{u_{-}},\\
&\{(x,\dot{\gamma}_{x,t}(t)):x\in M,\gamma_{x,t} \text{ is a minimizer of }T_{t}\varphi(x)\}=\pi\mathcal{L}\bar{\mathcal{J}}^{1}_{T_{t}\varphi}=\mathcal{L}_{T_{t}\varphi}\bar{\mathcal{G}}_{T_{t}\varphi},
\end{split}
\end{equation}

By the compactness of $\mathcal{L}_{u_{-}}\bar{\mathcal{G}}_{u_{-}}$ and the second step, we have for any $x\in M$, any minimizer $\gamma_{x,-}$ of $u_{-}(x)$, there exists $x^{\prime}\in M$ and a minimizer $\gm_{x^{\prime},t}$ of $T_{t}\varphi(x^{\prime})$ such that
\begin{equation}\label{3-3}
d_{TM}((x,\dot{\gamma}_{x,-}(0)),(x^{\prime},\dot{\gm}_{x^{\prime},t}(t)))=d_{TM}((x,\dot{\gamma}_{x,-}(0)),\mathcal{L}_{T_{t}\varphi}\bar{\mathcal{G}}_{T_{t}\varphi})\leq O(e^{-\frac{\lambda t}{3}}),
\end{equation}
thus $d_{M}(x,x^{\prime})\leq d_{TM}((x,\dot{\gamma}_{x,-}(0)),(x^{\prime},\dot{\gm}_{x^{\prime},t}(t)))\leq O(e^{-\frac{\lambda t}{3}})$. We estimate as
\begin{align*}
&d_{TM\times\R}((x,u_{-}(x),\dot{\gamma}_{x,-}(0)),\mathcal{L}\bar{\mathcal{J}}^{1}_{T_{t}\varphi})\\
\leq&d_{TM\times\R}((x,u_{-}(x),\dot{\gamma}_{x,-}(0)),(x^{\prime},T_{t}\varphi(x^{\prime}),\dot{\gm}_{x^{\prime},t}(t)))\\
\leq&d_{TM}((x,\dot{\gamma}_{x,-}(0)),(x^{\prime},\dot{\gamma}_{x^{\prime},t}(t)))+|u_{-}(x)-T_{t}\varphi(x)|+|T_{t}\varphi(x)-T_{t}\varphi(x^{\prime})|\\
\leq&d_{TM}((x,\dot{\gamma}_{x,-}(0)),\mathcal{L}_{T_{t}\varphi}\bar{\mathcal{G}}_{T_{t}\varphi})+O(e^{-\lambda t})+O(e^{-\frac{\lambda t}{3}})\\
\leq&O(e^{-\frac{\lambda}{3}t}),
\end{align*}
where the second inequality uses the definition of $d_{TM\times\R}$ and the third one uses the facts that $\{T_{t}\varphi\}_{t\geq t_{c}}$ is equi-Lipschitzian and $d_{M}(x,x^{\prime})\leq O(e^{-\frac{\lambda t}{3}})$. In the same way, we have
$$
d_{TM\times\R}((x,T_{t}\varphi(x),\dot{\gamma}_{x,t}(t)),\mathcal{L}\bar{\mathcal{J}}^{1}_{u_{-}})\leq O(e^{-\frac{\lambda}{3}t}).
$$

Since $\cup_{t\geq t_{c}}\bar{\mathcal{J}}^{1}_{T_{t}\varphi}\cup\bar{\mathcal{J}}^{1}_{u_{-}}$ is contained in a compact subset of $T^{\ast}M$, we could translate the above two inequalities into $T^{\ast}M\times\R$ by $\mathcal{L}$, which combining with \eqref{3-1} implies
\begin{equation}
\begin{split}
\max_{z\in\bar{\mathcal{J}}^{1}_{u_{-}}}d_{T^{\ast}M\times\R}(z,\bar{\mathcal{J}}^{1}_{T_{t}\varphi})\leq O(e^{-\frac{\lambda}{3}t}),\\
\max_{z\in\bar{\mathcal{J}}^{1}_{T_{t}\varphi}}d_{T^{\ast}M\times\R}(z,\bar{\mathcal{J}}^{1}_{u_{-}})\leq O(e^{-\frac{\lambda}{3}t}).
\end{split}
\end{equation}
Thus by the definition \ref{hd}, we have for the Hausdorff metric $d_{H}$ defined by $d_{T^{\ast}M\times\R}$,
\begin{equation}
d_{H}(\bar{\mathcal{J}}^{1}_{T_{t}\varphi},\bar{\mathcal{J}}^{1}_{u_{-}})\leq O(e^{-\frac{\lambda}{3}t}).
\end{equation}
Note that $d_{T^{\ast}M\times\R}$ is induced by a Riemannian metric and any two Riemannian metric defined on a common compact manifold are equivalent, we are done.

\appendix
\section{On the counterexample}
Example (\ref{ex}) shows that the milder assumptions (H1)-(H2) and $\frac{\partial H}{\partial u}\geq 0$ is not sufficient to guarantee the exponential convergence of the solution semigroup $\{T_t\}_{t\geq0}$ as in Theorem \ref{main}. We shall give a detail deduction here and first we need the following
\begin{lemma}\label{diffe}
Let $u$ be a Lipschitz function on $M$ and $\gm:[0,t]\rightarrow M$ a minimizer of $T_{t}u(x)$, there holds
\[\frac{\partial L}{\partial v}(\gm(0),u(\gm(0)),\dot{\gm}(0))\in D^-u(\gm(0)),\]
where $D^-u(x)$ denotes the lower differential of $u$ at $x$.
\end{lemma}
\Proof
It suffices to prove the lemma for the case when $M$ is an open subset of $\mathbb{R}^n$. Since $u$ is Lipschitzian, we only need to show for $y\in \R^n$
\begin{equation}\label{inff}
\begin{split}
\frac{\partial L}{\partial v}(\gm(0),u(\gm(0)),\dot{\gm}(0))\cdot y
\leq \liminf_{h\rightarrow 0^+}\frac{u(\gm(0)+h y)-u(\gm(0))}{ h}.
\end{split}
\end{equation}

Define $\gm_h:[0,\eps]\rightarrow M$ by $\gm_h(\tau)=\gm(\tau)+\frac{\eps-\tau}{\eps}h y$. We have
\begin{equation}\label{2-1}
\begin{split}
&T_\eps u(\gm(\eps))-u(\gm(0)+h y)\leq \int^{\eps}_{0}L(\gm_h(\tau),T_\tau u(\gm_h(\tau)),\dot{\gm}_h(\tau))d\tau,\\
&T_\eps u(\gm(\eps))-u(\gm(0))= \int^{\eps}_{0}L(\gm(\tau),T_\tau u(\gm(\tau)),\dot{\gm}(\tau))d\tau.
\end{split}
\end{equation}
Since $T_tu(x)$ is Lipschitz continuous, there exists $\kappa>0$ such that for $\tau\in [0,\eps]$,
\begin{equation}\label{2-2}
|T_su(\gm_h(\tau))-T_su(\gm(\tau))|\leq\kappa|\gm_h(\tau)-\gm(\tau)|=\kappa\frac{\eps-\tau}{\eps}h\|y\|.
\end{equation}

Combining \eqref{2-1},\eqref{2-2}, we have
\begin{align*}
&\liminf_{h\rightarrow 0^+}\frac{u(\gm(0)+h y)-u(\gm(0))}{h}\\
\geq&\lim_{\eps\rightarrow0^+}\frac{1}{\eps}\int^{\eps}_{0}\bigg[\frac{\partial L}{\partial v}(\gm(\tau),T_\tau u(\gm(\tau)),\dot{\gm}(\tau))\cdot y\bigg]d\tau\\
-&|y|\lim_{\eps\rightarrow0^+}\int^{\eps}_{0}\bigg[|\frac{\partial L}{\partial x}(\gm(\tau),T_\tau u(\gm(\tau)),\dot{\gm}(\tau))|+\kappa|\frac{\partial L}{\partial u}(\gm(\tau),T_\tau u(\gm(\tau)),\dot{\gm}(\tau))|\bigg]d\tau\\
=&\frac{\partial L}{\partial v}(\gm(0),u(\gm(0)),\dot{\gm}(0))\cdot y,
\end{align*}
where for the inequality, we use the fact that $|\frac{\eps-\tau}{\eps}|\leq1$ and for the equality, we use that $\frac{\partial L}{\partial x}, \frac{\partial L}{\partial u},\frac{\partial L}{\partial v}$ are $C^{1}$ thus locally bounded. This proves \eqref{inff}.
\end{proof}

We know that for every minimizer $\gm_{x,t}$ of $T_{t}\varphi(x)$, $(x(\tau),u(\tau),p(\tau)):=z_{x,t}(\tau)$ must satisfy the characteristic equations (\ref{hjech}). Moreover, Lemma \ref{diffe} shows that for every $C^1$ initial data $\varphi$, $p(0)=D\varphi(\gm_{x,t}(0))$. We choose the initial data $\varphi\equiv-1$, it immediately follows that for every $(x,t)\in M\times\R$ and the minimizer $\gm_{x,t}:[0,t]\rightarrow M$ attaining the infimum in the definition of $T_t\varphi(x)$, $p(0)=D\varphi(\gm_{x,t}(0))\equiv0$. From the definition of $H_{0}$ and equation (\ref{hjech}),

\begin{align}\label{B-1}
\left\{
        \begin{array}{l}
        \dot{p}=-\frac{3}{2}\rho^{\prime}(u^{3})u^{2}p,\\
        \dot{u}=\frac{1}{2}(p^2-\rho(u^{3})).
         \end{array}
                         \right.
\end{align}
By $p(0)=0$ and the first equation above, we have $p(s)\equiv0$ and the minimizer $\gm_{x,t}$ of $T_t\varphi(x)$ is the constant curve on $x$, i.e $\gm_{x,t}(s)=x$ for any $s\in[0,t]$. Combining $p(s)\equiv0$ and $u(0)=\varphi(\gm(0))=-1$, the second equation above shows that $u(s)$ is strictly increasing with respect to $s$ and forever negative, thus $u(s)\in[-1,0)$.

Now the last equation of (\ref{B-1}) reads
\begin{equation}
\dot{u}=-\frac{1}{2}u^3
\end{equation}
with the initial data $u(0)=\varphi\equiv-1$. Solving the above equation, we obtain that $u(t)=-(1+t)^{-\frac{1}{2}}\in[-1,0]$ when $t\geq0$. Moreover, we have
\begin{equation*}
u(x,t)=T_t\varphi(x)=T_t\varphi(\gm_{x,t}(t))=u(t)=-(1+t)^{-\frac{1}{2}},
\end{equation*}
which is independent of $x$.

\section{Admissible value set}
This appendix is devoted to a short introduction of the notion of admissible value set $\mathcal{C}_{H}$. We recall that a $C^{r}$ function $h:T^{\ast}M\rightarrow\R$ is called an autonomous classical Tonelli Hamiltonian if $H(x,u,p)=h(x,p)$ satisfies (H1)-(H2). As in \cite[Theorem A]{CIPP}, we can associate a real number
\begin{equation}\label{def-cv}
c(h)=\inf_{u\in C^\infty(M,\R)}\sup_{x\in M}h(x,\partial_xu)
\end{equation}
to the classical Tonelli Hamiltonian $h$ such that the equation $h(x,\partial_{x}u)=c$ has a solution if and only if $c=c(h)$. From \eqref{def-cv}, it follows that
\begin{proposition}\label{cv}
Let $h$ and $h_{i},i=1,2$ be autonomous classical Tonelli Hamiltonians, then
\begin{enumerate}
  \item  $c(h)$ is continuous with respect to the $C^{0}$ norm on the function space $C^0(T^{\ast}M,\R)$;
  \item  For any fixed real number $a\in\R$, $c(h+a)=c(h)+a$;
  \item  If $h_1\leq h_2$ on $T^{*}M$, then $c(h_1)\leq c(h_2)$.
\end{enumerate}
\end{proposition}

Now for a contact Hamiltonian $H$ and a given number $a\in\R$, set $h^{a}:T^{\ast}M\rightarrow\R$ by $h^{a}(x,p):=H(x,a,p)$. Assumptions (H1)-(H2) implies that for every $a\in\R$, $h^{a}$ is an autonomous classical Tonelli Hamiltonian and
\begin{definition}\label{adv}
The admissible value set $\mathcal{C}_H$ of $H$ is defined by
\begin{equation}
\mathcal{C}_H=\cup_{a\in\R}\hspace{0.1cm}c(h^{a}).
\end{equation}
\end{definition}

As a direct consequence of Definition \ref{adv}, we deduce some topological properties of the admissible value set $\mathcal{C}_{H}$.
\begin{proposition}\label{adv1}
Let $H$ be a contact Hamiltonian satisfying (H1)-(H3), then $\mathcal{C}_{H}\subseteq\R$ is an open interval. Moreover, if $\frac{\partial H}{\partial u}$ has a positive lower bound, then $\mathcal{C}_H=\R$.
\end{proposition}

\begin{proof}
By definition \ref{adv}, $\mathcal{C}_H$ is a non-empty subset of $\R$. It follows from the assumption (H3) that the map $j:\R\rightarrow C^0(T^{\ast}M,\R); j(a):=h^a$ is continuous with respect to the $C^0$-norm on $C^0(T^{\ast}M,\R)$. Thus $\mathcal{C}_H=c\circ j(\R)$ is an interval by (1) of Proposition \ref{cv}. By the fact that $\frac{\partial H}{\partial u}\geq0$ and (3) of Proposition \ref{cv}, the map $c\circ j$ is also monotone increasing.

Assume $\frac{\partial H}{\partial u}>0$ everywhere, we show that $c\circ j:\R\rightarrow\R$ is strictly increasing, this implies $\mathcal{C}_H$ is open. Fix two real numbers $a<a^{\prime}$, let $I:=[a,a^{\prime}], B:=c\circ j(a^{\prime})+1$, by (H1)-(H3) the set
$$
K_{I,B}=\{(x,u,p)\in T^{\ast}M\times\R|u\in I,|H(x,u,p)|\leq B\}
$$
is compact. Let $\lambda=\inf_{K_{I,B}}\frac{\partial H}{\partial u}>0$, we have that for $0<\epsilon<\min\{1,\lambda(a^{\prime}-a)\}$, there exists $u_{a^{\prime}}\in C^{\infty}(M,\R)$ such that
\begin{equation*}
\sup_{x\in M} H(x,a^{\prime},d_{x}u_{a^{\prime}}(x))<c\circ j(a^{\prime})+\epsilon.
\end{equation*}
Thus by definition of $K_{I,B}$, for any $x\in M, u\in I$, $(x,u,d_{x}u_{a^{\prime}}(x))\in K_{I,B}$ and
\begin{equation*}
c\circ j(a)\leq\sup_{x\in M} H(x,a,d_{x}u_{a^{\prime}})\leq\sup_{x\in M} H(x,a^{\prime},d_{x}u_{a^{\prime}})-\lambda(a^{\prime}-a)<c\circ j(a^{\prime}).
\end{equation*}

Assume $\lambda=\inf_{T^{\ast}M\times\R}\frac{\partial H}{\partial u}>0$, we show that $c(h^{a})\rightarrow\pm\infty$ as $a\rightarrow\pm\infty$ respectively. This fact and the monotone increasing property of $c\circ j$ implies $\mathcal{C}_H=\R$. By the assumption we have
\begin{itemize}
  \item for $a\geq 0$, $h^a\geq h^0+\lambda a$,
  \item for $a<0$, $h^a\leq h^0+\lambda a$.
\end{itemize}

According to (ii) and (iii) of Proposition \ref{cv}, we have
\begin{itemize}
  \item $c(h^{a})\geq c(h^{0}+\lambda a)=c(h^{0})+\lambda a\rightarrow +\infty$ as $a\rightarrow+\infty$,
  \item $c(h^{a})\leq c(h^{0}+\lambda a)=c(h^{0})+\lambda a\rightarrow -\infty$ as $a\rightarrow-\infty$.
\end{itemize}
This completes the proof.
\end{proof}

\section{Uniqueness of the stationary solution}
This section is devoted to a dynamical proof of the uniqueness of the solution of \eqref{sta} under the assumption (H1)-(H3) and $0\in\mathcal{C}_{H}$. See \cite{Ba3} for a proof from PDE aspects.

\begin{proof}
By Proposition \ref{prre-2}, stationary solutions are fixed points of $T_{t}$, thus it suffices to show that for any $t>0$ and any two distinct $\varphi,\psi\in C^{0}(M,\R)$,
\begin{equation}\label{A-1}
\|T_t\varphi-T_t\psi\|_{C^{0}}<\|\varphi-\psi\|_{C^{0}}.
\end{equation}
Set $a=\|\varphi-\psi\|_{C^{0}}$, \eqref{A-1} is equivalent to
\begin{equation}
T_t\psi-a<T_t\varphi<T_t\psi+a.
\end{equation}

Set $\psi_{\pm}=\psi\pm a$, it is clear that
\begin{equation}\label{A-2}
\begin{split}
\psi_{-}\leq\varphi\leq\psi_{+},\\
\psi_{-}<\psi<\psi_{+}.
\end{split}
\end{equation}
By Proposition \ref{prre-1} (i), we have
\begin{equation}\label{A-3}
\begin{split}
T_{t}\psi_{-}\leq T_{t}\varphi\leq T_{t}\psi_{+},\\
T_{t}\psi_{-}\leq T_{t}\psi\leq T_{t}\psi_{+}.
\end{split}
\end{equation}

By \eqref{A-2} and the continuity of $T_{t}$ with respect to $t$, there exists $\delta>0$ such that for $0\leq t\leq\delta$,
\begin{equation}\label{A-4}
T_{t}\psi_{-}<T_{t}\psi<T_{t}\psi_{+}.
\end{equation}
We choose a minimizer $\gamma_{x,t}$ of $T_{t}\psi(x)$, then \eqref{A-3}, \eqref{A-4} and the assumption (L3) implies
\begin{equation}\label{A-5}
\begin{split}
\int_0^\delta L(\gm_{x,t}(\tau),T_\tau\psi_{+}(\gm_{x,t}(\tau)),\dot{\gm}_{x,t}(\tau))d\tau<\int_0^\delta L(\gm_{x,t}(\tau),T_\tau\psi(\gm_{x,t}(\tau)),\dot{\gm}_{x,t}(\tau))d\tau\\
\int_\delta^{t}L(\gm_{x,t}(\tau),T_\tau\psi_{+}(\gm_{x,t}(\tau)),\dot{\gm}_{x,t}(\tau))d\tau\leq\int_\delta^{t} L(\gm_{x,t}(\tau),T_\tau\psi(\gm_{x,t}(\tau)),\dot{\gm}_{x,t}(\tau))d\tau.
\end{split}
\end{equation}

Inequalities \eqref{A-5} lead to
\begin{align*}
&T_{t}\psi_{+}(x)\leq\psi_{+}(\gm_{x,t}(0))+\int_0^tL(\gm_{x,t}(\tau),T_\tau\psi_{+}(\gm_{x,t}(\tau)),\dot{\gm}_{x,t}(\tau))d\tau\\
<&\psi(\gm_{x,t}(0))+a+\int_0^tL(\gm_{x,t}(\tau),T_\tau\psi(\gm_{x,t}(\tau)),\dot{\gm}_{x,t}(\tau))d\tau=T_{t}\psi(x)+a,
\end{align*}
where the first and last inequalities use Definition \ref{sg}. Combining \eqref{A-3}, we have $T_{t}\varphi\leq T_{t}\psi_{+}<T_{t}\psi+a$. Similarly, $T_{t}\psi-a<T_{t}\psi_{-}\leq T_{t}\varphi$ and this completes the proof.
\end{proof}

%\begin{remark}
%From the proof, we see that the assumption (H3) is not necessary, it could be replaced by the assumption that $H$ is uniformly Lipschitz and strictly increasing with respect to $u$.
%\end{remark}

\section{Convergence of discounted solutions}
In terms of (\ref{hbaa}), we consider the discounted Hamiltonian $H(x,u,p)$ formed as $H(x,u,p)=\lambda u+h(x,p)$, where $h(x,p):=-\lambda u_{-}(x)+H(x,u_{-}(x),p)$ and $\lambda$ is a positive real number. From the assumptions (H1)-(H3) on $H$, it follows that $h$ is an autonomous classical Tonelli Hamiltonian. We can refer \cite{dfiz,IMT1,IMT2} for the detailed analysis of this kind Hamiltonians.

 Note that $u_-$ is the  viscosity solution of  \eqref{sta} with the discounted Hamiltonian $H(x,u,p)=\lambda u+h(x,p)$. Moreover, there holds
\begin{proposition}\label{dvsr}
Let $H(x,u,p)=\lambda u+h(x,p)$ be the discounted Hamiltonian and $u(t,x),u_{-}(x)$ be the solutions of the equations \eqref{evo} and \eqref{sta} respectively, then
\begin{equation}\label{dvs}
\begin{split}
u(t,x)&=\inf_{\gm(0)=x}\left\{e^{-\lambda t}\varphi({\gm}(-t))+\int^0_{-t} e^{\lambda\tau}l({\gm}(\tau),\dot{{\gm}}(\tau))d\tau\right\},\\
u_{-}(x)&=\inf_{\gm(0)=x}\left\{\int^0_{-\infty} e^{\lambda\tau}l({\gm}(\tau),\dot{{\gm}}(\tau))d\tau\right\}.
\end{split}
\end{equation}
where
$$
l(x,\dot{x})=\sup_{p\in T^{*}_{x}M}\{\langle p,\dot{x}\rangle-h(x,p)\}
$$
is the convex dual of $h$ and both infimums in the above formula are taken among absolutely continuous curves.
\end{proposition}

An immediate corollary of the above representation is
\begin{lemma}\label{dsep}
Let $u(t,x),u_{-}(x)$ be the  solutions of the equations \eqref{evo} and \eqref{sta} with the discounted Hamiltonian $H(x,u,p)=\lambda u+h(x,p)$ respectively, then for $t\geq0$,
\begin{equation}\label{epd}
|u(t,x)-u_{-}(x)|\leq O(e^{-\lambda t}).
\end{equation}
\end{lemma}

\begin{proof}
By Proposition \eqref{dvsr}, there is an absolutely continuous curve $\gamma_{x,-}:(-\infty,0]\rightarrow M$ with $\gamma_{x,-}(0)=x$ that achieves the infimum in the formula \eqref{dvs} for $u_{-}$, then for $t\geq0$,
\begin{align*}
&u(t,x)-u_{-}(x)\\
\leq&e^{-\lambda t}\varphi({\gm}_{x,-}(-t))-\int^{-t}_{-\infty} e^{\lambda\tau}l({\gm}_{x,-}(\tau),\dot{{\gm}}_{x,-}(\tau))d\tau\\
\leq&e^{-\lambda t}\varphi({\gm}_{x,-}(-t))-\min_{TM}l(x,\dot{x})\int^{-t}_{-\infty} e^{\lambda\tau}d\tau\\
\leq&e^{-\lambda t}\bigg(\|\varphi\|_{C^0}-\frac{\min_{TM}l(x,\dot{x})}{\lambda}\bigg)\\
\leq &e^{-\lambda t}\bigg(\|\varphi\|_{C^0}+\frac{|\min_{TM}l(x,\dot{x})|}{\lambda}\bigg).
\end{align*}

By Proposition \eqref{dvsr}, there is an absolutely continuous curve $\gamma_{x,t}:[-t,0]\rightarrow M$ with $\gamma_{x,t}(0)=x$ that achieves the infimum in the formula \eqref{dvs} for $u(t,x)$, define $\xi:(-\infty,0]\rightarrow M$ by $\xi(\tau)=\gamma_{x,t}(\tau),\,\,\,\tau\in[-t,0]\,$ and $\xi(\tau)\equiv\gamma_{x,t}(-t),\,\,\,\tau\leq-t$, it follows that  $\xi$ is an absolutely continuous curve with $\xi(0)=x$ and
\begin{align*}
&u_{-}(x)-u(t,x)\\
\leq&e^{-\lambda t}|\varphi(\xi(-t))|+\int^{-t}_{-\infty}e^{\lambda\tau}l(\xi(\tau),\dot{\xi}(\tau))d\tau\\
\leq&e^{-\lambda t}|\varphi({\gm}_{x,t}(-t))|+e^{-\lambda t}\frac{l({\gm}_{x,t}(-t),0)}{\lambda}\\
\leq&e^{-\lambda t}\bigg(\|\varphi\|_{C^0}+\frac{\|l(\cdot,0)\|_{C^{0}}}{\lambda}\bigg).
\end{align*}
This completes the proof.
\end{proof}

\textbf{Acknowledgement}
Both authors would like to thank Prof. Jun Yan for his deep insight on this topic. L. Wang was partially under the support of National Natural Science Foundation of China (Grant 11631006), L. Jin was supported by the National Natural Science Foundation of China (Grants 11571166).

\end{document}